\documentclass[11pt]{amsart}

\usepackage[T1]{fontenc}
\usepackage{amsmath,amssymb,mathtools}
\usepackage{booktabs}
\usepackage{hyperref}

\hypersetup{
  colorlinks=true,
  linkcolor=blue,
  citecolor=blue,
  urlcolor=blue
}

\newcommand{\Hol}{\operatorname{Hol}}
\newcommand{\im}{\operatorname{im}}

\newcommand{\R}{\mathcal R}

\newcommand{\codim}{\operatorname{codim}}

\theoremstyle{plain}
\newtheorem{theorem}{Theorem}[section]
\newtheorem{proposition}[theorem]{Proposition}
\newtheorem{lemma}[theorem]{Lemma}
\newtheorem{corollary}[theorem]{Corollary}

\theoremstyle{definition}

\theoremstyle{remark}
\newtheorem{remark}[theorem]{Remark}

\title[Affine fixed points on flat manifold pairs]
  {Affine fixed points on flat manifold pairs}

\author{Aaron Reite}

\thanks{This paper was developed through an explicitly disclosed human--AI
research collaboration. Anthropic's Claude and OpenAI Codex participated
materially in the literature search, conjecture refinement, exact computations,
proof development, adversarial review, and manuscript preparation. The research
provenance is described in Section~\ref{sec:ai}.}

\date{September 15, 2026}


\begin{document}

\begin{abstract}
In this paper we consider affine selfmaps of compact flat manifold pairs.  We
derive an explicit formula for the relative Nielsen number in terms of the
ambient holonomy and the holonomy carried by the stabilizer of the lifted
submanifold.  The formula extends the flat specialization of Wong's
nilmanifold-pair formula to finite holonomy, agrees with the author's earlier
model-solvmanifold-pair formula on the flat overlap, and extends beyond it.
We also show that the
stabilizer holonomy cannot in general be replaced by the intrinsic holonomy of
the submanifold.  Examples exhibit interesting Schirmer theory and include
flat pairs outside the model-solvmanifold class.
\end{abstract}

\maketitle

\section{Introduction}

Let $f:(X,A)\to (X,A)$ be a selfmap of a compact polyhedral pair.  Schirmer
defined the relative Nielsen number by
\[
 N(f;X,A)=N(f)+N(f|_A)-N(f,f|_A),
\]
where the last term counts the essential fixed point classes of $f$ which
contain an essential fixed point class of the restriction
\cite{Schirmer1986}.  This number is a homotopy invariant and is a lower bound
for the minimum number of fixed points among maps homotopic to $f$ through
maps of pairs.  Under Schirmer's bypass and dimension hypotheses it is the
best possible lower bound.

The relative number has been computed on several classes of pairs.  Wong
proved an Anosov-type theorem for nilmanifold pairs
\cite{Wong1995}.  Cardona and Wong related the relative Nielsen and
Reidemeister numbers for Jiang-type pairs \cite{CardonaWong2001}.  Jezierski's
relative coincidence theory contains a finite-cover averaging theorem which
specializes to fixed points \cite[Corollary~3.3]{Jezierski1996}.  In a previous
paper we computed the relative Nielsen number for diagonal selfmaps of model
solvmanifold pairs \cite{Reite2008}.  Those manifolds have fundamental group
$\mathbb Z^n\rtimes_C\mathbb Z^p$, and the formula in
\cite[Theorem~4.5]{Reite2008} is indexed by cokernels of the base map and
filtered by determinants involving the gluing matrices $A^b$ in the notation
of that paper; the letter $A$ denotes the submanifold here.  Some model
solvmanifolds are flat, but the model construction neither supplies all
Bieberbach groups nor presents its relative term in intrinsic flat-manifold
data.  The purpose of the present paper is to give a closed, coordinate-free
determinant formula for affine maps of compact flat manifold pairs.  On the
overlap, the ambient stabilizer group $H_A$ below identifies the holonomy data
implicitly retained by the determinant filter in the earlier formula.  No
orientability assumption is required.

Absolute averaging computes the ambient and restriction Nielsen numbers
separately.  The relative number additionally requires identifying which
essential restriction classes lie in essential ambient classes and
controlling their possible fusion.  Proposition~\ref{prop:injection}
supplies this control.  Compatible torus covers then identify the correction
as a filtered sum over the ambient stabilizer holonomy $H_A$, giving an
explicit formula that applies beyond the earlier model pairs.

Let $\Pi\subset \mathbb R^n\rtimes O(n)$ be a Bieberbach group with
translation lattice $\Lambda$ and holonomy group $\Phi=\Pi/\Lambda$.  Write
\[
 M=\Pi\backslash\mathbb R^n.
\]
Let $\widetilde A=A_0+s$ be an affine subspace whose $\Pi$-translates are
pairwise disjoint or equal, and suppose that
\[
 A=\Pi_A\backslash\widetilde A,
 \qquad
 \Pi_A=\operatorname{Stab}_{\Pi}(\widetilde A),
\]
is compact.  Put
\[
 \Lambda_A=\Pi_A\cap\Lambda,
 \qquad H_A=\Pi_A/\Lambda_A.
\]
The group $H_A$ is naturally a subgroup of $\Phi$.  We call it the
\emph{ambient stabilizer holonomy}.  It need not be the intrinsic holonomy
group of $A$.

Let $F:(M,A)\to(M,A)$ be induced by an affine map
\[
 \widetilde F=(d,D):\mathbb R^n\longrightarrow\mathbb R^n
\]
chosen so that $\widetilde F(\widetilde A)\subseteq\widetilde A$.  If
$\alpha\in\Pi$, let $V_\alpha$ denote its linear part.  The main result is the
following.

\begin{theorem}\label{thm:intro}
Let $\Pi\subset\mathbb R^n\rtimes O(n)$ be a Bieberbach group with
translation lattice $\Lambda$ and holonomy group $\Phi=\Pi/\Lambda$, and put
$M=\Pi\backslash\mathbb R^n$.  Let $\widetilde A=A_0+s$ be an affine
subspace whose $\Pi$-translates are pairwise disjoint or equal.  Suppose that
$A=\Pi_A\backslash\widetilde A$ is compact, where
$\Pi_A=\operatorname{Stab}_\Pi(\widetilde A)$, and put
$\Lambda_A=\Pi_A\cap\Lambda$ and $H_A=\Pi_A/\Lambda_A$.  Let
$F:(M,A)\to(M,A)$ be induced by an affine lift
$\widetilde F=(d,D)$ satisfying
$\widetilde F(\widetilde A)\subseteq\widetilde A$.  Then
\begin{equation}\label{eq:main-intro}
\begin{split}
 N(F;M,A)
 ={}&\frac{1}{|\Phi|}
       \sum_{V\in\Phi}\left|\det(I-VD)\right| \\
 &+\frac{1}{|H_A|}
       \sum_{\substack{V\in H_A\\ \det(I-VD)=0}}
       \left|\det\left(I-(VD)|_{A_0}\right)\right|.
\end{split}
\end{equation}
\end{theorem}

Here and below a member of $\Phi$ or $H_A$ is represented by its orthogonal
linear part.  Since $V A_0=A_0$ for $V\in H_A$ and $D A_0\subseteq A_0$, the
restricted determinant in \eqref{eq:main-intro} is defined.  The first sum is
the usual averaging formula for $N(F)$.  The second sum counts the essential
restriction classes which lie in an inessential ambient sector.  Thus only
the ambient sectors with zero determinant contribute a relative correction.

When $\Phi=H_A=1$, Theorem~\ref{thm:intro} reduces to
\[
 N(F;M,A)=
 \begin{cases}
  |\det(I-D)|,&\det(I-D)\ne0,\\[2mm]
  |\det(I-D|_{A_0})|,&\det(I-D)=0.
 \end{cases}
\]
This is the flat, connected-submanifold case of Wong's nilmanifold-pair
formula \cite[Theorem~2.3(1)]{Wong1995}.  The new feature in
Theorem~\ref{thm:intro} is the finite-holonomy filter in the second sum.

The proof uses the averaging and mod-cover class results of Kim, Lee, and Lee
\cite{KimLeeLee2005}.  A map-dependent invariant lattice gives compatible
torus covers of $M$ and $A$.  The essential restriction classes are counted
on the connected torus cover of $A$ and then mapped directly to their ambient
sectors.  This avoids an auxiliary passage through the possibly disconnected
full inverse image of $A$.

In Section~2 we recall the relative Nielsen number and the required
finite-cover results.  Section~3 establishes the invariant covers and the
essential-class injection.  Theorem~\ref{thm:intro} is proved in Section~4.
Section~5 compares two flat overlap cases directly with \cite{Reite2008} and
then gives examples outside the model-solvmanifold class.  The invariant-axis
example shows that $H_A$ cannot be replaced by $\Hol(A)$, a four-dimensional
variant has $1\ne H_A\subsetneq\Phi$, and a product example has nonabelian
holonomy.  A final singular example illustrates the need for an invariant
sublattice of the translation lattice.

\section{Relative classes and finite covers}

Let $p:\widetilde X\to X$ be the universal cover of a compact connected
polyhedron, and let $\Pi$ be its group of covering transformations.  Fix a
lift $\widetilde f$ of a selfmap $f:X\to X$.  The relation
\[
 \widetilde f\alpha=\varphi(\alpha)\widetilde f,
 \qquad \alpha\in\Pi,
\]
defines an endomorphism $\varphi:\Pi\to\Pi$.  The fixed point classes of $f$
are represented by the lifts $\alpha\widetilde f$.  Two such lifts represent
the same class precisely when
\[
 \alpha'=\delta\alpha\varphi(\delta)^{-1}
\]
for some $\delta\in\Pi$.  We write $\R(\varphi)$ for the resulting set of
twisted conjugacy classes.

Suppose now that $f:(X,A)\to(X,A)$ is a map of pairs, with $A$ connected.
There is a function from the nonempty fixed point classes of $f|_A$ to the
fixed point classes of $f$, obtained by sending each restriction class to
the ambient class containing it.  Following Schirmer, let $N(f,f|_A)$ be the
number of essential ambient classes which contain an essential restriction
class.  Then
\begin{equation}\label{eq:schirmer}
 N(f;X,A)=N(f)+N(f|_A)-N(f,f|_A).
\end{equation}

The following elementary reformulation will be used in the proof of the main
theorem.

\begin{proposition}\label{prop:death-count}
Suppose that distinct essential fixed point classes of $f|_A$ have distinct
ambient images whenever that image is essential.  Then
\[
 \mathcal D_A(f)=
 \#\{\text{essential restriction classes with inessential ambient image}\}
\]
satisfies
\begin{equation}\label{eq:death-count}
 N(f;X,A)=N(f)+\mathcal D_A(f).
\end{equation}
\end{proposition}

\begin{proof}
Under the hypothesis, $N(f,f|_A)$ is exactly the number of essential
restriction classes having essential ambient image.  Subtract this number
from $N(f|_A)$ in \eqref{eq:schirmer}.
\end{proof}

We next record the finite-cover facts needed below.  Let $K\triangleleft\Pi$
be a finite-index subgroup such that $\varphi(K)\subseteq K$, and put
$Q=\Pi/K$.  The map $f$ has a lift $\bar f$ to $K\backslash\widetilde X$.
When $X$ is an infra-nilmanifold and the cover is a compact nilmanifold, Kim,
Lee, and Lee prove
\begin{equation}\label{eq:kll-average}
 N(f)=\frac{1}{|Q|}\sum_{\beta\in Q}N(\beta\bar f).
\end{equation}
See \cite[Theorem~3.5]{KimLeeLee2005}.  If the cover is a torus and
$\beta\bar f$ is induced by an affine map with linear part $E_\beta$, then
\begin{equation}\label{eq:torus}
 N(\beta\bar f)=|\det(I-E_\beta)|;
\end{equation}
see \cite{BrooksBrownPakTaylor1975}.

There is also a localized form of \eqref{eq:kll-average}.  Let
$\bar\varphi:Q\to Q$ be the induced endomorphism and let
$s=[\beta]\in\R(\bar\varphi)$.  The mod-$K$ fixed point class indexed by $s$
is a union of ordinary fixed point classes downstairs.  If
$N(\beta\bar f)>0$, the number of essential ordinary classes in this union is
\begin{equation}\label{eq:kll-local}
 \frac{|s|}{|Q|}N(\beta\bar f).
\end{equation}
If $N(\beta\bar f)=0$, the ordinary classes in the mod-$K$ class are
inessential.  This is Proposition~3.9 and Remark~3.7 of
\cite{KimLeeLee2005}.  Notice that \eqref{eq:kll-local} counts fixed point
classes downstairs, not fixed points on the cover.

\section{Affine flat pairs}

Let $\Pi\subset\mathbb R^n\rtimes O(n)$ be a Bieberbach group.  We write its
elements as $\alpha=(a_\alpha,V_\alpha)$.  Its translation subgroup
$\Lambda$ is a full lattice in $\mathbb R^n$, and
$\Phi=\Pi/\Lambda$ is finite.  Let $\widetilde A=A_0+s$ be as in the
introduction.  The disjoint-or-equal condition on its $\Pi$-translates makes
$A=\Pi_A\backslash\widetilde A$ an embedded flat submanifold of
$M=\Pi\backslash\mathbb R^n$.

The intrinsic translation subgroup of $A$ is
\[
 T_A=\{\alpha\in\Pi_A:V_\alpha|_{A_0}=I\}.
\]
Consequently
\[
 \Hol(A)=\Pi_A/T_A,
\]
whereas the subgroup occurring in Theorem~\ref{thm:intro} is
\[
 H_A=\Pi_A/\Lambda_A,
 \qquad \Lambda_A=\Pi_A\cap\Lambda.
\]
The inclusion $\Lambda_A\subseteq T_A$ need not be an equality.  The quotient
$T_A/\Lambda_A$ can retain transverse linear information which is invisible
in $\Hol(A)$.

Every affine pair map has a lift preserving the selected component.

\begin{lemma}\label{lem:normalize}
Let $F:(M,A)\to(M,A)$ be affine and let $\widetilde F=(d,D)$ be an affine
lift.  There is a $\gamma\in\Pi$ such that
$\gamma^{-1}\widetilde F(\widetilde A)\subseteq\widetilde A$.  If two lifts
preserve $\widetilde A$, their linear parts differ by left multiplication by
an element of $H_A$.
\end{lemma}

\begin{proof}
First note that the distinct translates of $\widetilde A$ are locally finite.
Indeed, compactness of $A$ gives a compact set $C\subset\widetilde A$ with
$\Pi_A C=\widetilde A$.  If a translate $\gamma\widetilde A$ meets a compact
set $K\subset\mathbb R^n$, then $\gamma\eta^{-1}C$ meets $K$ for some
$\eta\in\Pi_A$; properness of the $\Pi$-action leaves only finitely many such
elements $\gamma\eta^{-1}$, and hence only finitely many distinct translates.
The disjoint-or-equal translates are therefore open and closed in their union.
The connected set $\widetilde F(\widetilde A)$ consequently lies in one of
them, say $\gamma\widetilde A$.  Postcomposition by $\gamma^{-1}$ does not
change the map downstairs and gives the first assertion.  Two lifts of the
same map differ by postcomposition with a member of $\Pi$.  If both preserve
$\widetilde A$, the corresponding deck transformation belongs to $\Pi_A$.
\end{proof}

Changing a preserving lift multiplies its linear part by a member of $H_A$.
This merely permutes both sums in \eqref{eq:main-intro}.  Thus the expression
in that theorem is independent of the preserving lift.

The translation lattice need not be invariant under the endomorphism induced
by $\widetilde F$.  The following construction supplies the required
invariant subgroup.

\begin{lemma}\label{lem:core}
Let $\varphi:\Pi\to\Pi$ be an endomorphism.  Then
\begin{equation}\label{eq:core}
 L=\Lambda\cap\bigcap_{k\geq1}\varphi^{-k}(\Lambda)
\end{equation}
is a finite-index normal subgroup of $\Pi$, is contained in $\Lambda$, and
satisfies $\varphi(L)\subseteq L$.
\end{lemma}

\begin{proof}
For every $k\geq0$, the subgroup $\varphi^{-k}(\Lambda)$ is normal and
\[
 [\Pi:\varphi^{-k}(\Lambda)]
 =|\im(\Pi\xrightarrow{\varphi^k}\Pi\longrightarrow\Pi/\Lambda)|
 \leq |\Phi|.
\]
Since $\Pi$ is finitely generated, it has only finitely many subgroups of
bounded index.  Hence the intersection in \eqref{eq:core} is a finite
intersection of finite-index normal subgroups.  It is therefore finite-index
and normal, and it is contained in $\Lambda$ by definition.

If $x\in L$, then $\varphi^j(x)\in\Lambda$ for every $j\geq0$.  The same
statement for $\varphi(x)$ proves that $\varphi(x)\in L$.
\end{proof}

If $D$ is invertible, equivariance gives
$\varphi(t_\lambda)=t_{D\lambda}$ for every $\lambda\in\Lambda$.  Thus the
full translation lattice is already $\varphi$-invariant, and one may take
$L=\Lambda$.  The core construction is needed only when $D$ is singular.

Set
\begin{equation}\label{eq:cover-groups}
 S=L\backslash\mathbb R^n,
 \quad Q=\Pi/L,
 \quad L_A=L\cap\Pi_A=L\cap\Lambda_A,
 \quad Q_A=\Pi_A/L_A.
\end{equation}
The endomorphism $\varphi$ preserves $\Pi_A$.  Indeed, for
$\alpha\in\Pi_A$, equivariance gives
\[
 \varphi(\alpha)\widetilde F(\widetilde A)
 =\widetilde F\alpha(\widetilde A)
 =\widetilde F(\widetilde A).
\]
The nonempty set on the right lies in both $\widetilde A$ and
$\varphi(\alpha)\widetilde A$.  The disjoint-or-equal hypothesis therefore
gives $\varphi(\alpha)\in\Pi_A$.
The space $S$ is a torus.  Compactness of $A$ says that $\Pi_A$ acts
cocompactly on $\widetilde A$.  Since $L_A$ has finite index in $\Pi_A$ and
consists of translations, it also acts cocompactly and is a full lattice in
$A_0$.  Thus
\[
 \widehat A_0=L_A\backslash\widetilde A
\]
is also a torus.  Moreover, $L_A\triangleleft\Pi_A$ and
$\varphi(L_A)\subseteq L_A$.  The equality $L_A=L\cap\Pi_A$ gives a natural
injection
\begin{equation}\label{eq:sector-injection}
 \jmath:Q_A\longrightarrow Q.
\end{equation}

The next proposition is the fixed-class injection required by
Proposition~\ref{prop:death-count}.

\begin{proposition}\label{prop:injection}
Let $F:(M,A)\to(M,A)$ be affine and choose a lift $\widetilde F$ preserving
$\widetilde A$.  Distinct essential fixed point classes of $F|_A$ have
distinct ambient images whenever the ambient image is essential.
\end{proposition}

\begin{proof}
Let $\varphi:\Pi\to\Pi$ be induced by $\widetilde F$.  Suppose
$a,c\in\Pi_A$ represent two restriction classes which have the same ambient
class.  Then
\[
 a=\delta c\varphi(\delta)^{-1}
\]
for some $\delta\in\Pi$, and equivariance gives
\begin{equation}\label{eq:affine-conjugacy}
 a\widetilde F=\delta(c\widetilde F)\delta^{-1}.
\end{equation}

Suppose that the ambient class is essential.  The corresponding torus-sector
determinant is nonzero by \cite[Remark~3.7]{KimLeeLee2005}, so
$c\widetilde F$ has a unique fixed point $x_0$ in $\mathbb R^n$.  Since
$c\widetilde F$ preserves $\widetilde A$, its linear part preserves $A_0$.
The determinant factors over $A_0$ and
$\mathbb R^n/A_0$; hence the restricted determinant is also nonzero.  The
restriction of $c\widetilde F$ to $\widetilde A$ therefore has a fixed point,
which must be $x_0$.  Thus $x_0\in\widetilde A$.

By \eqref{eq:affine-conjugacy}, the unique fixed point of
$a\widetilde F$ is $\delta x_0$.  The same argument puts this point in
$\widetilde A$, while $x_0\in\widetilde A$ gives
$\delta x_0\in\delta\widetilde A$.  Hence
\[
 \widetilde A\cap\delta\widetilde A\ne\varnothing.
\]
The translates are disjoint or equal, so $\delta\widetilde A=\widetilde A$
and $\delta\in\Pi_A$.  The two classes are therefore already twisted
conjugate in $\Pi_A$.
\end{proof}

\begin{remark}
Neither the invertibility of $D$ nor invariance of $A_0$ under the full
holonomy group is used in Proposition~\ref{prop:injection}.  Only the
particular affine maps $a\widetilde F$ and $c\widetilde F$ must preserve the
selected affine subspace.
\end{remark}

\section{The determinant formula}

We now prove Theorem~\ref{thm:intro}.  The proof separates the absolute
ambient term from the restriction classes whose ambient sectors are
inessential.

\begin{proof}[Proof of Theorem~\ref{thm:intro}]
The preserving lift in the statement is supplied by
Lemma~\ref{lem:normalize}.  Let $\varphi$ be the induced endomorphism.  By
Lemma~\ref{lem:core}, we may choose $L$, $Q$, $L_A$, and $Q_A$ as in
\eqref{eq:core} and \eqref{eq:cover-groups}.  Let $\bar F:S\to S$ be the
induced map on the torus cover, and let
$\bar f:\widehat A_0\to\widehat A_0$ be its restriction.  Denote by
$\bar\varphi_A:Q_A\to Q_A$ the endomorphism induced by $\varphi|_{\Pi_A}$.
By \eqref{eq:kll-average} and \eqref{eq:torus},
\begin{equation}\label{eq:ambient-Q}
 N(F)=\frac{1}{|Q|}\sum_{\alpha\in Q}
       |\det(I-V_\alpha D)|.
\end{equation}
Every fiber of $Q\to\Phi$ has $[\Lambda:L]$ elements, and the determinant
depends only on the linear part.  Therefore
\begin{equation}\label{eq:ambient-Phi}
 N(F)=\frac{1}{|\Phi|}\sum_{V\in\Phi}|\det(I-VD)|.
\end{equation}

Put $f=F|_A$.  Proposition~\ref{prop:injection} and
Proposition~\ref{prop:death-count} give
\begin{equation}\label{eq:relative-death}
 N(F;M,A)=N(F)+\mathcal D_A(F).
\end{equation}

Apply \eqref{eq:kll-local} to the regular torus cover
$\widehat A_0\to A$.  If $s=[\beta]\in\R(\bar\varphi_A)$ and
\[
 n_A(\beta)=
 \left|\det\left(I-(V_\beta D)|_{A_0}\right)\right|>0,
\]
then $N(\beta\bar f)=n_A(\beta)$, and the mod-$L_A$ class indexed by $s$
contains
\begin{equation}\label{eq:localized-A}
 \frac{|s|}{|Q_A|}n_A(\beta)
\end{equation}
ordinary essential fixed point classes of $f$.  When $n_A(\beta)=0$, that
mod-cover class contains no essential ordinary class.

The injection $\jmath:Q_A\to Q$ identifies the ambient sector of every class
counted in \eqref{eq:localized-A}.  Indeed, the lift representatives in the
restriction sector have the form $k\beta$ with $k\in L_A\subset L$, and
hence all map to the twisted class of $\jmath(\beta)$ in
$\R(\bar\varphi)$.  Twisted conjugate elements of $Q_A$ map to twisted
conjugate elements of $Q$.  Remark~3.7 and the proof of Corollary~3.8 in
\cite{KimLeeLee2005} show that an essential mod-$L$ class consists entirely of
essential ordinary classes; the proof of their Theorem~3.5 shows that when the
covering Nielsen number is zero the corresponding ordinary classes are
inessential.  Together with the torus formula \eqref{eq:torus}, this says that
all ordinary ambient classes in the present sector are essential if
\[
 N(\jmath(\beta)\bar F)=|\det(I-V_\beta D)|>0,
\]
and they are inessential if this determinant vanishes.  Thus the classes in
\eqref{eq:localized-A} have inessential ambient image exactly when
$\det(I-V_\beta D)=0$.

Both determinant conditions are constant on the twisted class of $\beta$.
Indeed, equivariance gives
$DV_\delta=V_{\varphi(\delta)}D$.  Consequently, if
$\beta'=\delta\beta\bar\varphi_A(\delta)^{-1}$, then
$V_{\beta'}D=V_\delta(V_\beta D)V_\delta^{-1}$, and the same conjugacy holds
after restriction to $A_0$.  This argument does not require $D$ to be
invertible.

Summing over the twisted classes with inessential ambient sector, and then
replacing the class sum by the corresponding element sum, gives
\begin{equation}\label{eq:dead-QA}
 \mathcal D_A(F)=\frac{1}{|Q_A|}
   \sum_{\substack{\beta\in Q_A\\\det(I-V_\beta D)=0}}
   \left|\det\left(I-(V_\beta D)|_{A_0}\right)\right|.
\end{equation}
The projection $Q_A\to H_A$ has fibers of size
$[\Lambda_A:L_A]$.  Both determinants in \eqref{eq:dead-QA} depend only on
the full ambient linear part.  Hence \eqref{eq:dead-QA} becomes
\begin{equation}\label{eq:dead-HA}
 \frac{1}{|H_A|}
 \sum_{\substack{V\in H_A\\\det(I-VD)=0}}
 \left|\det\left(I-(VD)|_{A_0}\right)\right|.
\end{equation}
Combining \eqref{eq:ambient-Phi}, \eqref{eq:relative-death}, and
\eqref{eq:dead-HA} proves the theorem.
\end{proof}

The same argument gives the absolute restriction and common terms.

\begin{corollary}\label{cor:three-terms}
With the hypotheses of Theorem~\ref{thm:intro},
\begin{align*}
 N(F|_A)
 &=\frac{1}{|H_A|}\sum_{V\in H_A}
   \left|\det\left(I-(VD)|_{A_0}\right)\right|,\\
 N(F,F|_A)
 &=\frac{1}{|H_A|}
   \sum_{\substack{V\in H_A\\\det(I-VD)\ne0}}
   \left|\det\left(I-(VD)|_{A_0}\right)\right|.
\end{align*}
\end{corollary}

The preceding formulas give a direct analogue of the criterion in
\cite[Corollary~4.4]{Reite2008} for when Schirmer theory is interesting.

\begin{corollary}\label{cor:interesting}
With the hypotheses of Theorem~\ref{thm:intro},
\[
 N(F;M,A)>\max\{N(F),N(F|_A)\}
\]
if and only if both of the following conditions hold:
\begin{enumerate}
\item there is a $V\in H_A$ such that
\[
 \det(I-VD)=0
 \quad\text{and}\quad
 \det\left(I-(VD)|_{A_0}\right)\ne0;
\]
\item
\[
 N(F)>
 \frac{1}{|H_A|}
 \sum_{\substack{V\in H_A\\\det(I-VD)\ne0}}
 \left|\det\left(I-(VD)|_{A_0}\right)\right|.
\]
\end{enumerate}
\end{corollary}

\begin{proof}
By Theorem~\ref{thm:intro}, the strict inequality
$N(F;M,A)>N(F)$ holds exactly when the nonnegative correction term is
positive, which is condition~(1).  By \eqref{eq:schirmer}, the inequality
$N(F;M,A)>N(F|_A)$ is equivalent to $N(F)>N(F,F|_A)$; Corollary
\ref{cor:three-terms} identifies the latter term with the sum in condition
(2).
\end{proof}

\begin{remark}\label{rem:zero-ambient}
If $N(F)=0$, every summand in the nonnegative ambient average vanishes.  The
common term in Corollary~\ref{cor:three-terms} is then zero, and
$N(F;M,A)=N(F|_A)$.  This recovers Schirmer's general observation
\cite[Theorem~2.5(i)]{Schirmer1986} and, on the flat model-pair overlap,
case~(i) of \cite[Theorem~4.5]{Reite2008}.
\end{remark}

\begin{corollary}\label{cor:minimum}
Suppose, in addition, that $\dim M\geq3$ and $\codim A\geq2$.  Then
\[
 MF[F;M,A]=N(F;M,A),
\]
and the common value is given by \eqref{eq:main-intro}.
\end{corollary}

\begin{proof}
Schirmer's four hypotheses are that the ambient polyhedron is connected, its
complement has no local cut point and is not a $2$-manifold, every component
of the subpolyhedron is a Nielsen space, and the subpolyhedron can be
bypassed \cite[Theorem~6.2]{Schirmer1986}.  We verify them below.

The quotient $M$ is connected.  The compact embedded submanifold $A$ is
closed, and a compatible triangulation makes $(M,A)$ a pair of compact
polyhedra.  Since $M-A$ is a manifold of dimension at least three, it has no
local cut point and is not a $2$-manifold.  General position shows that the
locally flat submanifold $A$ can be bypassed when its codimension is at least
two.

It remains to verify the Nielsen-space hypothesis.  The manifold $A$ is
connected.  If $\dim A\geq2$, then it has no local cut point and is not a
surface of negative Euler characteristic; in dimension two it is a torus or
a Klein bottle.  Jiang's Nielsen-space criterion therefore applies
\cite{Jiang1980}; see also Schirmer's summary
\cite[p.~466]{Schirmer1986}.  For the torus, see also
\cite{BrooksBrownPakTaylor1975}.  If $A=S^1$, a degree-$m$ map with $m\ne1$
is homotopic to $z\mapsto z^m$, which has $|1-m|$ fixed points, and conjugating
by a circle homeomorphism places them at any prescribed distinct points.  A
nontrivial rotation realizes the fixed-point-free degree-one case.  A
zero-dimensional connected $A$ is a point and is immediate.  Thus every
component of $A$ is a Nielsen space, and the Minimum Theorem applies.
\end{proof}

\section{Examples}

We say that a pair map has \emph{interesting Schirmer theory} if
\[
 N(F;M,A)>\max\{N(F),N(F|_A)\}.
\]
The following examples illustrate both the correction term and the role of
the ambient stabilizer holonomy.

\subsection{An invariant axis}

Let $\Pi$ be the orientable three-dimensional Bieberbach group of the
half-turn manifold, also called the dicosm or $\mathcal G_2$, generated by
the translations $t_1,t_2,t_3$ and
\[
 g=\left(\tfrac12e_1,V\right),
 \qquad V=\operatorname{diag}(1,-1,-1),
 \qquad g^2=t_1.
\]
Put $M=\Pi\backslash\mathbb R^3$ and
$\widetilde A=\mathbb R e_1$.  Then
\[
 \Pi_A=\langle g\rangle,
 \qquad \Lambda_A=\langle g^2\rangle,
 \qquad H_A\cong\mathbb Z/2.
\]
However, $g$ restricts to translation by $e_1/2$ on the axis.  Thus
$T_A=\Pi_A$, $A$ is a circle, and $\Hol(A)=1$.

Consider the affine map with preserving lift $\widetilde F(x)=Dx$, where
\[
 D=\operatorname{diag}(-1,1,1).
\]
It induces $g\mapsto g^{-1}$ and therefore descends to $M$.
The two ambient determinants are
\[
 \det(I-D)=0,
 \qquad \det(I-VD)=8,
\]
so $N(F)=4$.  Both restricted determinants equal $2$, and hence
$N(F|_A)=2$.  Only the identity ambient sector is inessential.  Theorem
\ref{thm:intro} gives
\[
 N(F;M,A)=4+\frac12(2)=5.
\]
Equivalently, the two restriction classes are represented by the even and
odd powers of $g$; the even class maps to the inessential identity sector and
the odd class maps to the essential $V$-sector.  Thus $N(F,F|_A)=1$.

If $H_A$ were incorrectly replaced by the intrinsic holonomy of the circle,
the formula would give $4+2=6$.  This example therefore detects the transverse
ambient information in $H_A$.  Corollary~\ref{cor:minimum} also gives
\[
 MF[F;M,A]=5>4=\max\{N(F),N(F|_A)\}.
\]

This pair is also the flat model pair covered by
\cite[Corollary~4.2]{Reite2008}: in the notation of that paper, the dicosm is
$\mathbb Z^2\rtimes_{-I_2}\mathbb Z$, the submanifold is the canonical base
circle, and the map has $X=I_2$ and $Y=-1$.  Hence
$\operatorname{Coker}(I-Y)=\mathbb Z/2$.  Its even term has
$\det(I-X)=0$ and contributes the correction $1$, while its odd term has
$|\det(I+X)|=4$ and contributes to the ambient number.  The earlier formula
therefore also gives $4+1=5$.  In this coordinate description, the old
$\delta$-filter already uses the ambient matrix $(-I_2)^bX$ rather than data
intrinsic to the base circle.  Theorem~\ref{thm:intro} identifies the group
carrying that filter as $H_A$ and does so without model-solvmanifold
coordinates.

\subsection{A second model-pair cross-check}

There is a second useful comparison with \cite[Example~4.1]{Reite2008}.  Take
the flat model solvmanifold
\[
 \Pi=\mathbb Z^3\rtimes_{-I_3}\mathbb Z,
 \qquad
 \Lambda=\mathbb Z^3\times2\mathbb Z,
\]
acting on $\mathbb R^3\times\mathbb R$ by
$(u,k)\cdot(x,t)=(u+(-I_3)^k x,t+k)$.  Its holonomy has representatives
$I$ and $V=\operatorname{diag}(-I_3,1)$.  Set
$\widetilde A=A_0=\operatorname{span}(e_3,e_4)$ and let $F$ be induced by
the preserving lift $\widetilde F(z)=Dz$, where
\[
 D=X\oplus7,
 \qquad
 X=\begin{pmatrix}
 1&0&0\\
 0&1&0\\
 1&0&-5
 \end{pmatrix}.
\]
Then $A$ is a Klein bottle and
$H_A=\Hol(A)=\mathbb Z/2$.  The ambient sector determinants are
\[
 |\det(I-D)|=0,
 \qquad
 |\det(I-VD)|=96,
\]
so $N(F)=48$.  On $A_0$, the corresponding determinants are $36$ and $24$,
respectively.  Corollary~\ref{cor:three-terms} and
Theorem~\ref{thm:intro} therefore give
\[
 N(F|_A)=\frac{36+24}{2}=30,
 \qquad
 N(F;M,A)=48+\frac{36}{2}=66.
\]
These are exactly the values obtained from the independently proved model-pair
formula in \cite[Example~4.1]{Reite2008}.  Thus the two theories agree in a
nontrivial flat overlap with both a live and a dead ambient sector.

\subsection{A four-dimensional orientable pair}

Let $\Lambda=\mathbb Z^4$ and let $\Pi$ be generated by $\Lambda$ and the
glides
\[
 g_B=\left((\tfrac12,0,0,0),B\right),
 \qquad
 g_C=\left((0,\tfrac12,0,0),C\right),
\]
where
\[
 B=\operatorname{diag}(1,-1,1,-1),
 \qquad
 C=\operatorname{diag}(1,1,-1,-1).
\]
Then $M=\Pi\backslash\mathbb R^4$ is orientable and has holonomy
$\Phi=\{I,B,C,BC\}$.  Take
\[
 \widetilde A=\operatorname{span}(e_1,e_2)
 +(0,0,\tfrac14,\tfrac13).
\]
Its stabilizer consists of the translations in
$\operatorname{span}(e_1,e_2)$, so $A$ is a $2$-torus and $H_A=1$.

Let $F$ have preserving lift $\widetilde F(x)=Dx$, where
\[
 D=\operatorname{diag}(-1,-1,1,1).
\]
This map normalizes the generators above and preserves $\widetilde A$.
The ambient sector determinants, ordered as $I,B,C,BC$, are
\[
 0,\quad0,\quad16,\quad0.
\]
Therefore $N(F)=4$.  The restriction determinant is $4$, and its ambient
identity sector is inessential.  It follows that
\[
 N(F;M,A)=4+4=8.
\]
This example already lies outside the model-solvmanifold class of
\cite{HeathKeppelmann2002,Reite2008}.  Indeed, if its fundamental group had a
model presentation $\mathbb Z^n\rtimes_C\mathbb Z^p$, projection to
$\mathbb Z^p$ would give $b_1(M)\geq p$.  But
\[
 H_1(M)=\mathbb Z\oplus(\mathbb Z/2)^3,
\]
so $p\leq1$.  In the flat case the normal abelian subgroup $\mathbb Z^n$ lies
in the Bieberbach translation subgroup.  Indeed, if $N\triangleleft\Pi$ is
abelian and $g=(a,V)\in N$, then
$t_\lambda g t_\lambda^{-1}g^{-1}=t_{(I-V)\lambda}\in N$ for every
$\lambda\in\Lambda$.  Commutativity with $g$ gives
$(I-V)^2\lambda=0$.  Since $\Lambda$ spans $\mathbb R^4$ and the
finite-order matrix $V$ is diagonalizable over $\mathbb C$, this forces
$V=I$.  The holonomy is therefore a quotient
of $\Pi/\mathbb Z^n\cong\mathbb Z^p$ and is generated by at most $p$ elements.
It would be cyclic, contrary to $\Phi\cong(\mathbb Z/2)^2$.  Thus this example
tests the endpoint $H_A=1$ beyond the model class, even though the ambient
holonomy is abelian.  The same homology calculation also shows that $M$ is not
a product of two flat surfaces.  Finally,
\[
 N(F;M,A)=8>4=\max\{N(F),N(F|_A)\}.
\]
Corollary~\ref{cor:minimum} gives $MF[F;M,A]=8$.

The same manifold also supports a map for which the correction comes entirely
from a nonidentity member of a nontrivial proper stabilizer group.  Replace the
affine subspace by
\[
 \widetilde A'=\operatorname{span}(e_1,e_2)
 +(0,0,\tfrac14,0).
\]
The $B$-glide preserves $\widetilde A'$, whereas the $C$- and $BC$-cosets
send it to a plane whose third coordinate is congruent to $-1/4$, not $1/4$,
modulo $\mathbb Z$.
Thus
\[
 \Pi_{A'}=\langle t_1,t_2,g_B\rangle,
 \qquad
 H_{A'}=\{I,B\}\cong\mathbb Z/2\subsetneq\Phi.
\]
The other translates are the parallel planes with transverse coordinates
$(\pm1/4,0)+\mathbb Z^2$, so they are disjoint or equal.  The glide $g_B$
restricts to $(x,y)\mapsto(x+1/2,-y)$; hence $A'$ is a Klein bottle and
$H_{A'}=\Hol(A')$.

Let $F':(M,A')\to(M,A')$ be induced by the affine lift
$\widetilde F'(x)=d'+D'x$, where
\[
 D'=\operatorname{diag}(-1,3,3,-1),
 \qquad d'=(0,0,-\tfrac12,0).
\]
It preserves $\widetilde A'$ because the third coordinate of its offset is
$3(1/4)-1/2=1/4$.  It also descends to $M$: translations are sent according
to $D'$, and the induced endomorphism satisfies
\[
 \varphi'(g_B)=g_B^{-1},
 \qquad
 \varphi'(g_C)=g_Ct_{(0,1,1,0)}.
\]
The ambient determinants for $I,B,C,BC$ are now
\[
 16,\quad0,\quad0,\quad64,
\]
so $N(F')=20$.  The restricted determinants for $I$ and $B$ are
\[
 \left|\det(I-D'|_{A_0})\right|=4,
 \qquad
 \left|\det(I-(BD')|_{A_0})\right|=8.
\]
Thus the identity sector is live, the $B$-sector is dead, and the entire
relative correction comes from the nonidentity element $B$.  Consequently
\[
 N(F'|_{A'})=6,
 \qquad
 N(F',F'|_{A'})=2,
 \qquad
 N(F';M,A')=20+\frac{8}{2}=24.
\]
This is interesting Schirmer theory, and Corollary~\ref{cor:minimum} gives
$MF[F';M,A']=24$.  If one wrongly included every dead sector in $\Phi$ and
divided the correction by $|\Phi|$, the result would be
$20+(8+4)/4=23$.  Omitting the nonidentity $B$-sector would instead give
$20$.  Thus this variant exercises the filter, index set, and normalization
in the second sum with $1\ne H_{A'}\subsetneq\Phi$.

\subsection{A nonabelian-holonomy example}

We finish with an example having nonabelian holonomy.  Let
$(M_{\mathrm{ax}},A_{\mathrm{ax}})$ and $F_{\mathrm{ax}}$ be the invariant-axis
example above.  Let $B_4=O^4_{21}$ be the
orientable closed flat $4$-manifold listed in Table~3 of
\cite{LambertRatcliffeTschantz2025}.  Its holonomy is the dihedral group
$D_4$ of order eight, and
\[
 H_1(B_4)=\mathbb Z\oplus(\mathbb Z/2)^2.
\]

There is an affine selfmap $G:B_4\to B_4$ with linear part $9I_4$ and which
induces the identity on $D_4$.  To see this, write the Bieberbach extension as
\[
 0\longrightarrow\Lambda_B\longrightarrow\Gamma_B
 \longrightarrow D_4\longrightarrow1
\]
with extension class $\xi\in H^2(D_4;\Lambda_B)$.  Positive-dimensional
cohomology of a finite group is annihilated by its order
\cite[Chapter~III]{Brown1982}, so $8\xi=0$ and $9\xi=\xi$.  The maps
$9I_4$ on $\Lambda_B$ and the identity on $D_4$ therefore induce an
endomorphism of $\Gamma_B$.  The affine realization theorem for
infra-nilmanifolds supplies $G$ \cite[Theorem~1.1]{Lee1995}.

The real holonomy representation of $D_4$ on $\mathbb R^4$ decomposes as
\[
 \mathbb R^4\cong\mathbf1\oplus\det\oplus\rho,
\]
where $\rho$ is the standard two-dimensional representation.  Indeed,
$b_1(B_4)=1$ gives one trivial summand, faithfulness forces the
two-dimensional irreducible summand, and orientability forces the remaining
character to be $\det\rho$.  If $r$ is a quarter rotation and $s$ a
reflection, then
\[
\begin{array}{c@{\qquad}c}
\toprule
 W&|\det(I-9W)|\\
\midrule
 1&4096\\
 r,r^3&5248\\
 r^2,s,rs,r^2s,r^3s&6400\\
\bottomrule
\end{array}
\]
and hence
\[
 N(G)=\frac{4096+2(5248)+5(6400)}8=5824.
\]

Now put
\[
 M=M_{\mathrm{ax}}\times B_4,
 \qquad A=A_{\mathrm{ax}}\times B_4,
 \qquad F=F_{\mathrm{ax}}\times G.
\]
The preserving lift has linear part
$D_{\mathrm{ax}}\oplus9I_4$, where
$D_{\mathrm{ax}}=\operatorname{diag}(-1,1,1)$, and
\[
 \Phi=H_A=(\mathbb Z/2)\times D_4.
\]
Since every $D_4$-sector of $G$ is essential, the two terms of
Theorem~\ref{thm:intro} factor as the corresponding terms for
$F_{\mathrm{ax}}$ times
$N(G)$.  Therefore
\[
 N(F)=4(5824)=23296
\]
and the relative correction is $5824$.  Thus
\[
 N(F;M,A)=29120.
\]
Corollary~\ref{cor:minimum} applies and gives
\[
 \boxed{N(F;M,A)=MF[F;M,A]=29120.}
\]

The holonomy $(\mathbb Z/2)\times D_4$ is nonabelian, so the preceding model
presentation argument also excludes this example from the model-solvmanifold
class.  It further lies outside the class of flat solvmanifolds, whose
holonomy groups are abelian \cite[Theorem~3.2]{Tolcachier2020}.  The
construction is deliberately a product, and its relative computation factors
accordingly.  Its distinct role is to show that the determinant formula
accommodates nonabelian holonomy; the four-dimensional examples above already
provide non-model tests with both $H_A=1$ and
$1\ne H_A\subsetneq\Phi$.

\subsection{A singular map requiring an invariant sublattice}

Take the dicosm of Section~5.1 crossed with a circle.  Its group $\Pi$ is
generated by $\Lambda=\mathbb Z^4$ and
\[
 g=(\tfrac12e_1,V),\qquad
 V=\operatorname{diag}(1,-1,-1,1),\qquad g^2=t_1.
\]
Set $\widetilde A=\mathbb R e_1$, so $A$ is a circle,
$\Pi_A=\langle g\rangle$, and $H_A=\Phi\cong\mathbb Z/2$.
The rank-one affine lift
\[
 \widetilde F(x_1,x_2,x_3,x_4)=(x_4/2,0,0,0)
\]
preserves $\widetilde A$ and satisfies equivariance with
\[
 \varphi(t_4)=g,\qquad
 \varphi(g)=\varphi(t_1)=\varphi(t_2)=\varphi(t_3)=1.
\]
In particular, $\varphi(\Lambda)\not\subseteq\Lambda$.  Since $\varphi^2$
is trivial, the core in Lemma~\ref{lem:core} is
\[
 L=\Lambda\cap\varphi^{-1}(\Lambda)
  =\mathbb Ze_1+\mathbb Ze_2+\mathbb Ze_3+2\mathbb Ze_4.
\]
Indeed, a lattice translation with fourth coordinate $m$ maps to $g^m$,
which is a translation exactly when $m$ is even.  Thus $[\Lambda:L]=2$,
$[\Pi:L]=4$, and $\varphi(L)\subseteq L$.

The linear part $D$ satisfies $D^2=0$ and $VD=D$, so both ambient
determinants are $1$.  The restriction is constant, and both restricted
determinants are also $1$.  Theorem~\ref{thm:intro} and
Corollary~\ref{cor:three-terms} give
\[
 N(F)=N(F|_A)=N(F,F|_A)=N(F;M,A)=1.
\]
There is also a direct check: $F^2$ is constant at the image of the origin,
so this is the unique fixed point of $F$.  The example exhibits a singular
affine pair map for which the full translation lattice cannot serve as the
invariant cover.

\paragraph{Computational verification.}
The displayed examples were checked using exact integer and rational
arithmetic.  The first two examples also agree with the independently proved
formula of \cite{Reite2008}, providing a separate check on the flat overlap.
An exploratory parameter sweep served only as an internal consistency check:
its reference calculation used the same sector rule as the theorem.
The verification scripts are not included with this paper; the proof and
displayed calculations do not require them.

\section{Declaration of generative AI and AI-assisted technologies in the
manuscript preparation process}\label{sec:ai}

This paper was intentionally developed as a disclosed human--AI research
collaboration between Aaron Reite, Anthropic's Claude, and OpenAI Codex.  The
retained record documents Claude's use in the early notebook and adversarial
stages and later hostile reviews, and Codex's use in source retrieval, proof
audit and compression, exact computations, drafting, and revision.  Their
roles overlapped in conjecture testing, proof development, source work,
computation, and exposition.  Both systems made substantial mathematical
contributions.

The collaboration materially changed the result.  In particular, it located
the finite-cover theorem of Jezierski, identified the flat specialization of
Wong's theorem as the exact trivial-holonomy boundary, separated the ambient
stabilizer holonomy from the intrinsic holonomy of $A$, and reduced the final
proof to the direct sector
map $Q_A\hookrightarrow Q$ together with Proposition~\ref{prop:injection} and
the mod-cover count of Kim, Lee, and Lee.

The AI systems also produced substantive false intermediate claims.  An early
version treated relative finite-cover averaging as new, some exploratory
fixed-class calculations omitted component twists, and an early formula
conflated $H_A$ with $\Hol(A)$.  These failures were retained in the research
record and used to design exact controls and a falsification checklist.

The research notes, claim ledger, scripts, and record of failed approaches
are retained privately and are not included with this paper.

Neither Claude nor Codex is listed as an author.  The human author retains
full responsibility for the statements and proofs.

\bibliographystyle{amsplain}
\bibliography{references}

\end{document}